\documentclass{amsart}

\newtheorem{theorem}{Theorem}[section]
\newtheorem{lemma}[theorem]{Lemma}

\theoremstyle{definition}

\newtheorem{proposition}[theorem]{Proposition}
\newtheorem{corollary}[theorem]{Corollary}

\theoremstyle{remark}

\numberwithin{equation}{section}

\newcommand{\interior}[1]{%
  {\kern0pt#1}^{\mathrm{o}}%
}

\usepackage{amssymb}
\usepackage{blindtext}
\usepackage{relsize}
\usepackage{geometry}
 \usepackage{hyperref}
 \hypersetup{hidelinks}
 \usepackage[symbol]{footmisc}

\begin{document}

\title[Duality theory and characterizations of optimal solutions for ...]{\textbf{Duality theory and characterizations of optimal solutions for a class of conic linear problems}}

\author[Nick Dimou]{Nick Dimou$^*$}
\footnote[0]{$^*$Prospective student; recently completed undergraduate studies at the National and Kapodistrian University of Athens, Department of Mathematics, Athens, Greece.\par Work was completed at the Game Theory Lab, School of Mathematical Sciences in Tel Aviv University, while a visiting student. Invitation was given by professor Eilon Solan of Tel Aviv University while still being an undergraduate student. All research facilities were provided by the same department. Besides professor Eilon Solan, for this enlightening academic opportunity that he offered and his extremely significant help and guidance, the author would also like to thank Panayotis Mertikopoulos, Alexander Shapiro and Rakesh Vohra for their comments and recommendations.}
\footnote[0]{\textit{email address}: dimou.nikolas@gmail.com}



\begin{abstract}
For a primal-dual pair of conic linear problems that are described by convex cones $S\subset X$, $T\subset Y$, bilinear symmetric objective functions $\langle\cdot,\cdot\rangle_X$, $\langle\cdot,\cdot\rangle_Y$ and a linear operator $A:X\rightarrow Y$, we show that the existence of optimal solutions $x^*\in S$, $y^*\in T$ that satisfy $Ax^*=b$ and $A^Ty^*=c$ eventually comes down to the consistency and solvability of the problems $min\langle z,z\rangle_Y,\;z\in\{Ax-b:x\in S\}$ and $ min\langle w,w\rangle_X,\; w\in\{A^Ty-c:y\in T\}$. Assuming that these two problems are consistent and solvable, strong duality theorems as well as geometric and algebraic characterizations of optimal solutions are obtained via natural generalizations of the Farkas' Lemma without a closure condition. 
Some applications of the main theory are discussed in the cases of continuous linear programming and linear programming in complex space.
\end{abstract}

\maketitle

\section{Introduction}
Optimization problems of the subsequent form have been profoundly studied and strong duality theorems under different frameworks have been acquired:
\begin{center}
\begin{equation*}{(P')\;\;\;\;\;\;\;\;}
\begin{array}{ll}
    min\;\; \langle c,x\rangle_X\\
    s.t.\;\; Ax-b\in T\\ \;\;\;\;\;\;\;\;x\in S
    \end{array}
\end{equation*}
\end{center}
\begin{center}
\begin{equation*}{(D')\;\;\;\;\;\;\;\;}
\begin{array}{ll}
    max\;\; \langle y,b\rangle_Y \\
    s.t.\;\; -A^Ty+c\in S^*\\ \;\;\;\;\;\;\;\;\;\;\;y\in T^*
    \end{array}
    \end{equation*}
\end{center} 
where the objective functions are usually defined over pairings $(X,X')$ and $(Y,Y')$. Such primal-dual pairs of problems are called conic linear problems when the objective functions are bilinear, $A:X\rightarrow Y$ is a linear map and $S,T$ are convex cones \cite{3}. The above pair of problems is often characterized by many authors as a generalized form of the following primal-dual pair of conic optimization problems to the infinite-dimensional case:
\begin{center}
\begin{equation*}{(P^*)\;\;\;\;\;\;\;\;}
\begin{array}{ll}
    min\;\; \langle c,x\rangle_X\\
    s.t.\;\; Ax=b\\ \;\;\;\;\;\;\;\;x\in S
    \end{array}
\end{equation*}
\end{center}
\begin{center}
\begin{equation*}{(D^*)\;\;\;\;\;\;\;\;}
\begin{array}{ll}
    max\;\; \langle y,b\rangle_Y \\
    s.t.\;\; -A^Ty+c\in S^*\\ \;\;\;\;\;\;\;\;\;\;\;y\in T
    \end{array}
    \end{equation*}
\end{center} 
Here the objecive functions usually represent inner products defined over a product of finite Euclidean spaces and the sets $S,$ $T$ are (polyhedral) convex cones.\par
For problems of both forms that are defined over finite-dimensional vector spaces, a rather complete theory of duality has been developed (see e.g. \cite{5, 12}). The infinite-dimensional case is not yet considered complete, however numerous results regarding strong duality have been obtained under general frameworks and assumptions. These assumptions over the feasible sets are known as \textit{constraint qualifications} (CQs) in the relevant literature, and in comparison to the standard linear programming (LP) case, they are more than often necessary for a zero duality gap to exist. The most well-known, sufficient of these constraint qualifications turns out to be that of \textit{strict feasibility}, or else known as \textit{Slater} CQ (or \textit{Slater's} CQ), which assumes that a feasible solution on the (relative) interior of the corresponding cone exists. In particular the following known result for the pair of problems $\{(P^*),(D^*)\}$ in the finite-dimensional case can be readily found in any textbook of semidefinite programming (for example \cite{5, 16}):\\ \\
\textbf{Theorem.}
    \textit{Suppose that there exists $y^*\in T^*$ such that $-A^Ty^*+c\in int(S^*)$. Then $val(P^*)=val(D^*)$ and there exists optimal feasible $x^*\in S$ such that $Ax^*=b$.}\\

\par

Various CQs have been proved to guarantee a zero duality gap between the primal and dual conic linear problems in the infinite dimensional case; Some early results are given by Kretschmer \cite{8} in 1961, where he treated problems of the form $\{(P'),(D')\}$ defined over paired vector spaces and closed convex cones with respect to the weak topologies through subconsistency and subvalues. A wide review of duality theory in infinite-dimensional linear programming was later given by Anderson \cite{1} and Nash and Anderson \cite{2}. A very interesting connection between conic linear problems and sensitivity analysis was also made by Shapiro \cite{15}. In his work strong duality results were obtained under the subdifferentiality of the optimal value function $v(y):=inf\{\langle c,x\rangle :x\in S,$ $Ax-y\in T\}$ in $b$. Results in a framework and problem setup similar to \cite{8} were recently given by Khanh et al. \cite{7}.\par

However, in all the above, 
strong duality results for the pair of conic linear problems $\{(P'),(D')\}$ are obtained in topological frameworks where a Slater CQ, when it does guarantee strong duality, does not necessarily guarantee the existence of optimal solutions $x^*\in S$, $y^*\in T$ that satisfy $Ax^*=b$ and $A^Ty^*=c$ respectively. In other words, while the complementarity conditions (slackness)
\begin{equation}
    \langle y^*,Ax^*-b\rangle_Y =0\;\;\mbox{and\;\;}  \langle c-A^Ty^*,x^*\rangle_X =0 
\end{equation}
\textit{strongly} hold (that is there exist $x^*\in S$ such that $Ax^*=b$) for the pair $\{(P^*),(D^*)\}$, they do not necessarily \textit{strongly} hold for the primal-dual pair of conic linear problems $\{(P'),(D')\}$.\par

In this paper we show that in a certain class of conic linear problems, under strict feasibility the complementarity conditions (slackness) (1.1) hold, and therefore strong duality between $(P')$ and $(D')$ also holds. In fact we show something stronger; if there exists an optimal solution $x^*$ such that it belongs to int$S$, then whether $Ax^*-b$ belongs to the interior of $T$ or not, there always exists feasible $y^*$ such that $A^Ty^*=c$; and the same goes for a feasible solution of $(D')$ in int$T^*$. That is we show that (1.1) \textit{strongly} holds for both the primal and the dual under some specific conditions.\par 

This particular class of conic linear problems is defined over a more general topological framework than that of conic convex programming problems of the form $(P^*),$ $(D^*)$, since the objective functions are not necessarily inner products (nor continuous), the vector spaces $X$, $Y$ are not necessarily Euclidean spaces (or subsets) and the convex cones $S,$ $T$ are not necessarily closed nor polyhedral. It should be mentioned that conic linear problems treated here belong to the wider class of problems that appear in previous work (\cite{7, 8, 15}) and therefore similar duality results are already known.\par
The key to our approach resides on the fact that the basic properties that constitute this class of conic linear problems are such that the proof of zero duality gap eventually comes down to the requirement that the two subsequent problems are consistent and solvable (that is that they have optimal feasible solutions):\vspace{0.25cm}
\begin{equation}
    min\langle z,z\rangle_Y,\;z\in\{Ax-b:x\in S\}\;\;\;\mbox{and\;\;\;} min\langle w,w\rangle_X,\; w\in\{A^Ty-c:y\in T\}
\end{equation}
In fact, it can be easily shown that (1.2) can be replaced be three other conditions, the alternations that involve the max$\langle z,z\rangle_Y$ and min$\langle w,w\rangle_X$ (that we shall call condition (1.2)a), min$\langle z,z \rangle_Y$ and max$\langle w,w\rangle_X$ (condition (1.2)b), and max$\langle z,z\rangle_Y$ and max$\langle w,w\rangle_X$ (condition (1.2)c), of the objective functions of these two convex sets respectively.\par
Since our aim is to prove the existence of points $x^*$ and $y^*$ such that $Ax^*=b$ and $A^Ty^*=c$ it is only logical that our approach should involve some kind of theorems similar to the Farkas alternative in real space. Indeed, for the proof of the two main results we use simple generalized versions of the Farkas' Lemma which hold without a closure assumption, in comparison to the trivial case of (LP) and other similar linear problems that require a closeness property for strong duality to hold. Instead, these generalized forms of the Farkas' Lemma derive from properties that constitute the class of conic linear problems treated, and from simple extended hyperplane theorems.\par

 The rest of the paper is organized as follows: in Section 2 the framework and class of conic linear problems are defined, the main results are formulated and some basic notations are given. The main section of this paper, Section 3, includes the proofs of the two key theorems as well as some other useful results regarding solutions that satisfy (1.1) and the consistency of some subsets of the feasible sets. In Section 4 some applications of the main theory are discussed in the case of linear programming in complex space. Lastly, in Section 5 we obtain a simple result regarding continuous linear programming problems which can be characterized as complementary to the strong duality results given by Grinold \cite{6} and Levinson \cite{9}.\\ \\

\section{A class of conic linear problems}

In this paper we deal with conic linear problems of the following form:
\begin{center}
\begin{equation*}{(P)\;\;\;\;\;\;\;\;}
\begin{array}{ll}
    min\;\; \langle c,x\rangle_X\\
    s.t.\;\; Ax-b\in T^*\\ \;\;\;\;\;\;\;\;x\in S
    \end{array}
\end{equation*}
\end{center}
\begin{center}
\begin{equation*}{(D)\;\;\;\;\;\;\;\;}
\begin{array}{ll}
    max\;\; \langle y,b\rangle_Y\\
    s.t.\;\; -A^Ty+c\in S^*\\ \;\;\;\;\;\;\;\;\;\;\;y\in T
    \end{array}
    \end{equation*}
\end{center}
\par Similarly to the linear restrictions that appear in conic programming theory $S,T$ are convex cones$^1$\footnote[0]{$^1$ A set $S\subset X$ is a convex cone iff $\forall\lambda\in[0,1]:$ $\lambda S+(1-\lambda)S\subset S$ and $\forall\mu\in[0,\infty):$ $\mu S\subset S$.}, subsets of vector spaces $X$ and $Y$ respectively and $S^*,T^*$ are their respective positive dual cones$^2$\footnote[0]{$^2$ The positive dual cone of $S$ is the convex set $S^*:=\{x^*\in X:\langle x^*,x\rangle_X\geq0\;\forall x\in S\}$.}. The objective functions $\langle\cdot,\cdot\rangle_X$, $\langle\cdot,\cdot\rangle_Y$ defined over $X\times X$ and $Y\times Y$ respectively, are bilinear and symmetric (although symmetry is not always necessary in the following). Here $A:X\rightarrow Y$ is a linear operator and $A^T:Y\rightarrow X$ is its adjoint$^3$\footnote[0]{$^3$ Some authors also denote the adjoint by $A^*$.}, $b\in Y$ and $c\in X$. In addition, the spaces $X,Y$ are not necessarily assumed to be Euclidean or Hilbert spaces, nor do the objective functions represent specifically defined inner products. Alike \cite{15}, the topologies corresponding to the vector spaces (if $X,\;Y$ are equipped with any) are abstract and the space $X$ is ``large" enough, such that the adjoint of the linear operator exists. That is, either the linear mapping $A$ is assumed continuous (usually represented by a matrix function), or space $X$ is ``large" enough such that ``for every $y\in Y$ there exists a unique $x^*\in X$ such that $\langle y,Ax\rangle =\langle x^*,x\rangle$ for every $x\in X$", so that the adjoint can be accordingly defined and $\langle y,Ax\rangle=\langle A^Ty,x\rangle$ for every $x\in S$ and every $y\in T$.\par

A noticable difference between the two structures of the restrictions of dual problems $(P'),(D')$ and $(P),(D)$ lies on the switching of the cones $T$ and $T^*$. Although this alternation might seem confusing to the reader that is familiar with the known work on duality theory of conic programming, it does not affect the gist of the following analysis nor the main purposes of the paper.\\ \par
The class of conic linear problems that is treated in the above form is characterized by the following three assumptions:\\ \\ \\
\textbf{(F1)} Convex cones $S,\;T$ are solid, that is int$S\neq\O$, int$T\neq\O$.\\
\textbf{(F2)} The positive-definiteness condition is met for the objective functions in the sets $(C_{A}-b)\times(C_{A}-b)$, $(D_{A}-c)\times(D_{A}-c)$, where $C_{A}:=\{Ax:x\in X\}\subset Y$ and $D_{A}:=\{A^{T}y:y\in Y\}\subset X$, that is $\langle z,z\rangle_Y > 0$ for every $0\neq z\in C_{A}-b$, $\langle w,w \rangle_X > 0$ for every $0\neq w\in D_{A}-c$.\\
\textbf{(F3)} The objective functions $\langle\cdot,\cdot\rangle_Y$, $\langle\cdot,\cdot\rangle_X$ acquire a minimum value in the sets $(C_A-b)\times(C_A-b)$, $(D_A-c)\times(D_A-c)$ respectively.\\ \\
Assumption \textbf{(F1)} means that we only deal with essential non-trivial conic linear problems. It is also essential for strictly feasible solutions to exist. Assumption \textbf{(F2)} generalizes the case of the objective functions being inner products, and \textbf{(F3)} implies that the closure of the convex cones $S$, $T$ and the continuity of the objective functions are not required.\par

Hypothesis \textbf{(F3)} essentially translates into the significant remark we made on the previous senction, that is that the existence of optimal solutions $x^*\in S$, $y^*\in T$ such that $Ax^*=b$ and $A^Ty^*=c$ eventually comes down to the consistency and solvability of the problems $min\langle z,z\rangle_Y,\;z\in\{Ax-b:x\in S\}$ and $ min\langle w,w\rangle_X,\; w\in\{A^Ty-c:y\in T\}$. In addition, it seems to be the most obscure of the three judging by the fact that the first two are met more than often, compared to \textbf{(F3)}, in the relevant literature. However, the latter is also met under some generic hypotheses and known results in variational analysis (e.g. see \cite{14} Theorem 1.9). Moreover, in the framework of \cite{5} \textbf{(F2)} and \textbf{(F3)} are also true for the closure of the sets $C_A$, $D_A$, therefore it is seen that the case of $S,T$ being closed overlaps already known results in finite (real) Euclidean spaces.\\ \par 

It must be pointed out that, due to the symmetric form of the primal-dual pair and the linear symmetry that appears in the proofs of the generalized forms of the Farkas' Lemma, this last assumption may be replaced by the acquirement of the maximum instead of the minimum, respectively. What this essentially means is that if assumptions \textbf{(F1), (F2)} hold as well as just one of the four different variations of assumption \textbf{(F3)} (conditions (1.2), (1.2)a, (1.2)b, (1.2)c), then the same main results presented here hold. In what follows we only treat with the case where the minimization of the two problems is assumed, in order to avoid repeating results.\par

Before we state the two main results of this paper, we give some basic notations that are used throughout the next sections:\\ \\
\par We denote by $\mathcal{S}(P):=\{x\in S:Ax-b\in T^*\}$ the set of all feasible solutions of the primal problem $(P)$. Respectively,  we denote by $\mathcal{S}(D):=\{y\in T:-A^Ty+c\in S^*\}$ the set of all feasible solutions of the dual problem $(D)$. The optimal value of the primal problem $(P)$ is defined by $v(P):=inf\{\langle c,x\rangle_X :Ax-b\in T^*$, $x\in S\}$ and the optimal value of the dual problem $(D)$ is defined by $v(D):=sup\{\langle y,b\rangle_Y : -A^Ty+c\in S^*$, $y\in T\}$, when of course the respective feasibility sets are non-empty (we define $v(P):=\infty$, $v(D):=-\infty$ if the problems have no feasible solutions). The following weak duality relation holds by the standard minimax duality:
\begin{equation}
    v(P)\geq v(D)
\end{equation}
\par In addition we define the sets $\tilde{\mathcal{S}}(P):=\{x\in\mathcal{S}(P):\langle c,x\rangle_X<\infty\}$ and $\tilde{\mathcal{S}}(D):=\{y\in\mathcal{S}(D):\langle y,b\rangle_Y>-\infty\}$ (we tend to use these sets of feasible solutions in order to avoid ``unsavory" optimization cases). We lastly denote by $\mathcal{S}^*(P)$ and $\mathcal{S}^*(D)$ the sets of all optimal feasible solutions of problems $(P)$ and $(D)$ respectively, that is the sets of all feasible solutions for which the minimum and maximum value is attained respectively.\par
Note that a problem might have an optimal value while not having an optimal solution (the interested reader may look at problem (4.10) in \cite{5}).
\par
We shall also from now on skip the indicators $``X",\;``Y"$ of the objective functions, and we will write $\langle \cdot,\cdot \rangle$ without provoking any confusion regarding the space that this is defined over.\\ \\ 

We now state the two main results that are to be obtained in the next section. Under assumptions \textbf{(F1)-(F3)} the following is true:
\begin{theorem}
If the problems $(P)$, $(D)$ have optimal solutions $x^*$, $y^*$ respectively such that $x^*\in$ int$S$, $y^*\in$ int$T$ and the optimal values are finite, then $v(P)=v(D)$. In fact, there exist optimal solutions $\hat{x}\in S$, $\hat{y}\in T$ of $(P)$, $(D)$ respectively such that $A\hat{x}=b$ and $A^T\hat{y}=c$.\\
\end{theorem}
Now assume the following: \textbf{(F4)} The sets $X\smallsetminus$int$S$ and $Y\smallsetminus$int$T$ are cones (not necessarily convex).\\ \par
Let the sets $\hat{\mathcal{S}}(D):=\{y\in$ int$T:-A^Ty+c\in S^*,-A^Ty\in S^*,\langle y,b\rangle>-\infty\}$, $\hat{\mathcal{S}}(P):=\{x\in$ int$S:Ax-b\in T^*,Ax\in T^*,\langle x,c\rangle<\infty\}$. If \textbf{(F4)} also holds for the above class of conic linear problems, then the following is true:
\begin{theorem}
    If $\hat{\mathcal{S}}(D)\neq\O$, $\hat{\mathcal{S}}(P)\neq\O$, $\tilde{\mathcal{S}}(P)\smallsetminus\hat{\mathcal{S}}(P)\neq\O$, $\tilde{\mathcal{S}}(D)\smallsetminus\hat{\mathcal{S}}(D)\neq\O$ and if  $v(P),\;v(D)$ are finite, then $v(P)=v(D)$. In fact, there exist optimal solutions $\hat{x}\in S$, $\hat{y}\in T$ of $(P)$, $(D)$ respectively such that $A\hat{x}=b$ and $A^T\hat{y}=c$. \\ 
\end{theorem}
Even though assumption \textbf{(F4)} ``shortens" the main class that is treated in this paper, it essentially releases us from the necessity of the existence of optimal solutions for both dual problems when specific feasible solutions can be found. The boundeness conditions for the optimal values and the upper and lower boundeness for the values of certain feasible solutions in the formulations of the two results could as well be omitted when their existence is known, while their presence insinuates that a class of functions that take infinite values over the cones $S,T$ is not excluded from this particular framework.\\ 
\par Now, one can straight away identify the geometric and algebraic characterizations that occur from the above results regarding the existence of optimal solutions:\\ \par
The algebraic characterization has a clear explanation as it directly occurs from both Theorems that there exist points $x^*\in S$, $y^*\in T$ that are solutions to the linear ``systems" $Ax=b$, $A^Ty=c$. The geometric characterizations is justified by the converse direction. To be more specific, if there exist optimal solutions of the dual problems with finite optimal values, but the two systems $Ax=b$, $x\in S$ and $A^Ty=c$, $y\in T$ do not have a solution, then every optimal solution of the two problems belongs to the boundary of the corresponding convex cone (when of course such set is non-empty). Same goes with Theorem 2.2., for in the case where the two systems have no feasible solutions, but the optimal values exist and are finite, then the subsets $\hat{\mathcal{S}}(P):=\{x\in$ int$S:Ax-b\in T^*$, $Ax\in T^*,\langle x,c\rangle<\infty\}$, $\hat{\mathcal{S}}(D):=\{y\in$ int$T:-A^Ty+c\in S^*$, $-A^Ty\in S^*,\langle y,b\rangle>-\infty\}$ of feasible solutions are empty. In other words, by knowing that the linear ``systems" $Ax=b$, $A^Ty=c$ have no solutions in $S$ and $T$ respectively, then we instantly know in which ``geometric" place to ``look for" the optimal solutions.\\ \\
\section{Duality theory and characterization of optimal solutions}
In this section the main results of the paper are proved. Theorem 2.1 derives from the following two key theorems, which can be characterized as dual due to their formulation and initial hypothesis.
\begin{theorem}
    If $\mathcal{S}^*(D)\;\cap$ int$T\neq\O$ and $v(D)<+\infty$ then there exists $\hat{x}\in\mathcal{S}(P)$ such that $A\hat{x}=b$.
\end{theorem}
\begin{theorem}
    If $\mathcal{S}^*(P)\;\cap$ int$S\neq\O$ and $v(P)>-\infty$ then there exists $\hat{y}\in\mathcal{S}(D)$ such that $A^T\hat{y}=c$.
\end{theorem}
Here we only prove Theorem 3.1; the proof of Theorem 3.2 is completely analogous, as are the proofs of the lemmas that are used in the following analysis.\par
In what follows we only deal with the class of conic linear problems of Section 2, that is the convex cones $S$, $T$ and the objective functions satisfy assumptions \textbf{(F1)-(F3)}.\\ 
\begin{lemma}
    Let $C$ be a convex subset of $X$ such that the positive-definiteness condition is met for the objective function over $C-b$ for some $b\in X\smallsetminus C$ and let $c\in C$. Then the following are equivalent:\\
    \begin{itemize}
        \item[(i)] $\langle c-b,c-b\rangle\leq\langle x,x\rangle$ $\forall x\in C-b$
        \item[(ii)] $\langle c-b,x-c\rangle\geq 0$ $\forall x\in C$
    \end{itemize}
\end{lemma}
\begin{proof}
    (i)$\Rightarrow$(ii): Let $x\in C$ and 
    $\lambda\in (0,1)$, then $c+\lambda(x-c)\in C$. We now have:
    \begin{equation*}
        \langle c-b+\lambda(x-c),c-b+\lambda(x-c)\rangle - \langle c-b,c-b\rangle = 2\lambda\langle c-b,x-c\rangle + \lambda^2\langle x-c,x-c\rangle
    \end{equation*}
    due to the linearity and symmetry of $\langle\cdot,\cdot\rangle$. By (i) $2\langle c-b,x-c\rangle + \lambda\langle x-c,x-c\rangle\geq 0$, hence for $\lambda=1/n$, $n\in\mathbb{N}$ we obtain (ii) as $\lambda\rightarrow 0$.\par
    (ii)$\Rightarrow$(i): Let $x\in C$. Then
    \begin{equation*}
        \langle x-b,x-b\rangle - \langle c-b,c-b\rangle = \langle x-c,x-c\rangle + 2\langle c-b,x-c\rangle\Rightarrow\langle x-b,x-b\rangle - \langle c-b,c-b\rangle\geq 0
    \end{equation*}
    by (ii) and the positive-definiteness condition.\\
\end{proof}
\begin{proposition}
    Let $C$ be a convex subset of $X$ and let $b\in X\smallsetminus C$ such that the positive-definiteness condition is met for the objective function over $C-b$. If $\langle\cdot,\cdot\rangle$ acquires a minimum value on $C-b$ then there exists $\alpha\in X$ such that $\langle\alpha,b\rangle< \langle\alpha,x\rangle$ $\forall x\in C$. In fact $\alpha\in C-b$.
\end{proposition}
\begin{proof}
Let $\gamma\in C$ such that $\langle\gamma-b,\gamma-b\rangle\leq\langle x-b,x-b\rangle$ $\forall x\in C$. Then by Lemma 3.3 and the positive-definiteness condition:
\begin{equation*}
    \langle\gamma-b,x\rangle\geq\langle\gamma-b,\gamma\rangle > \langle\gamma-b,b\rangle
\end{equation*}
since $\gamma\neq b$. The inequality $\langle\alpha,b\rangle< \langle\alpha,x\rangle$ holds for $\alpha:=\gamma-b\in C-b$.\\
\end{proof}
We now obtain the following generalized form of Farkas' Lemma:\\
\begin{theorem}(1st generalized form of Farkas' Lemma)
    Let $S\subset X$ be a convex cone, $S^*$ its positive dual and $b\in Y$. Then exactly one of the following problems has at least one solution:\\
\begin{equation*}
(I) \left. \begin{array}{ll}
         Ax=b\\
        x\in S \end{array} \right., \;\;\;
(II) \left. \begin{array}{ll}
         (-A)^Ty\in S^*\\
        \langle y,-b\rangle <0\\
        y\in Y \end{array} \right.
        \end{equation*}
\end{theorem}
\begin{proof}
    Suppose that both $(I)$ and $(II)$ have solutions. Let $\hat{x}\in S$ be a solution of $(I)$ and $\hat{y}\in Y$ be a solution of $(II)$. Then:
    \begin{equation*}
        (-A)^T\hat{y}\in S^*\Rightarrow\langle (-A)^T\hat{y},x\rangle\geq0\;\;\; \forall x\in S\Rightarrow \langle (-A)^T\hat{y},\hat{x}\rangle\geq 0\Rightarrow\langle \hat{y},-A\hat{x}\rangle\geq 0\Rightarrow \langle\hat{y},-b\rangle\geq 0
    \end{equation*}
    which is a contradiction.\par
    Now assume that $(I)$ does not have a solution. Then $b\not\in C_A=\{Ax:x\in S\}\subset Y$. Since $S$ is a convex cone and $A$ is linear, the set $C_A$ is also a convex cone. Therefore Proposition 3.4 can be applied for $C_{A}$ in space $Y$ as \textbf{(F2)} and \textbf{(F3)} hold; there exists $\alpha\in -C_A+b$ such that $\langle\alpha,-b\rangle<\langle\alpha,-Ax\rangle$ $\forall x\in S$. Let $x\in S$, then $\lambda x\in S$ $\;\forall\lambda\in(0,\infty)$ hence $\langle\alpha,-b\rangle/\lambda<\langle\alpha,-Ax\rangle$. Since $x\in S$ is arbitrary, for $\lambda=n\in\mathbb{N}$ we obtain
    \begin{equation*}
        \langle\alpha,-Ax\rangle\geq 0\;\; \forall x\in S\Rightarrow\langle -A^T\alpha,x\rangle\geq 0\;\; \forall x\in S\Rightarrow -A^T\alpha\in S^*
    \end{equation*}
    The inequality $\langle\alpha,-b\rangle<0$ holds by definition ($0\in S$), therefore $\alpha\in -C_A+b\subset Y$ is a solution of $(II)$.\\  
\end{proof}
The proof of Theorem 3.1 now follows:
\begin{proof}\textit{(Theorem 3.1)} Suppose by contradiction that a feasible solution $x\in S$ of $(P)$ such that $Ax=b$ does not exist. By Theorem 3.5 the following problem has a solution:
\begin{equation*}
    (II^*) \left. \begin{array}{ll}
         (-A)^Ty\in S^*\\
        \langle y,-b\rangle <0\\
        y\in Y \end{array} \right.
\end{equation*}
Suppose that for every solution $\hat{y}$ of $(II^*)$ there exists $(y^*,\lambda)\in(\mathcal{S}^*(D)\;\cap$\;int$T)\times(0,\infty)$ such that $y^*+\lambda\hat{y}\in$ int$T$. Then $y^*+\lambda\hat{y}\in\mathcal{S}(D)$.  Indeed:
\begin{equation*}
    \langle-A^Ty^*+c,x\rangle\geq 0 \mbox{ and}\;\langle-A^T\hat{y},x\rangle\geq0\;\;\forall x\in S\Rightarrow \langle -A^T(y^*+\lambda\hat{y})+c,x\rangle\geq 0\;\;\forall x\in S
\end{equation*}
Then $\langle y^*+\lambda\hat{y},b\rangle=\langle y^*,b\rangle+\lambda\langle\hat{y},b\rangle>\langle y^*,b\rangle$, which is a contradiction since $y^*$ is optimal and $v(D)<\infty$.\par
Therefore there exists a solution $\hat{y}$ of $(II^*)$ such that $\forall (y^*,\lambda)\in(\mathcal{S}^*(D)\;\cap$ int$T)\times(0,\infty)$ $y^*+\lambda\hat{y}\in Y\smallsetminus$int$T$. For $\lambda=1/n$, $n\in\mathbb{N}$ we obtain $y^*\in Y\smallsetminus$int$T$ for every $y^*\in\mathcal{S}^*(D)\;\cap$ int$T$ as $\lambda\rightarrow 0$ since $Y\smallsetminus$int$T$ is closed (with respect to the corresponding topology of $Y$).\\
\end{proof}
The proof of Theorem 3.2 is completely analogous to the above, and derives from the generalized form of the Farkas' Lemma, which is similarly formulated and proved in the topological vector space $Y$ as follows:\\
\begin{theorem}(2nd generalized form of Farkas' Lemma) Let $T\subset Y$ be a convex cone, $T^*$ its positive dual and $c\in X$. Then exactly one of the following problems has at least one solution:\\
\begin{equation*}
(I') \left. \begin{array}{ll}
         A^Ty=c\\
        y\in T \end{array} \right., \;\;\;
(II') \left. \begin{array}{ll}
         Ax\in T^*\\
        \langle x,c\rangle <0\\
        x\in X \end{array} \right.
        \end{equation*}
\end{theorem}

For the purposes of this paper, the proof is omitted (we leave the formulations and proofs of the corresponding Lemma 3.3 and Proposition 3.4 to the reader).\par
By Theorems 3.1 and 3.2 we obtain the following corollary which is the first of the two main results of this paper:
\begin{corollary}\textit{(Theorem 2.1)}
    If both $(P)$ and $(D)$ have optimal solutions $x^*$, $y^*$ such that $x^*\in$ int$S$, $y^*\in$ int$T$, and $v(P),\;v(D)$ are finite then there exist feasible solutions $\hat{x}\in S$, $\hat{y}\in T$ that satisfy $A\hat{x}=b$, $A^T\hat{y}=c$ and therefore $v(P)=v(D)$ and $\hat{x},$ $\hat{y}$ are optimal.\\ 
\end{corollary}
Similarly to Theorem 2.1, the second main result of this paper, Theorem 2.2, results from two key theorems when \textbf{(F4)} holds:
\begin{theorem}
    If $\hat{\mathcal{S}}(D)=\{y\in$ int$T:-A^Ty+c\in S^*,-A^Ty\in S^*,\langle y,b\rangle>-\infty\}\neq\O$, $\tilde{\mathcal{S}}(D)\smallsetminus\hat{\mathcal{S}}(D)\neq\O$ and $v(D)<\infty$ then there exists $\hat{x}\in\mathcal{S}(P)$ such that $A\hat{x}=b$. 
\end{theorem}
\begin{theorem}
    If $\hat{\mathcal{S}}(P)=\{x\in$ int$S:Ax-b\in T^*,Ax\in T^*,\langle x,c\rangle<\infty\}\neq\O$, $\tilde{\mathcal{S}}(P)\smallsetminus\hat{\mathcal{S}}(P)\neq\O$ and $v(P)>-\infty$ then there exists $\hat{y}\in\mathcal{S}(D)$ such that $A^T\hat{y}=c$.\\
\end{theorem}
Again, the proof of Theorem 3.9 is omitted and only that of Theorem 3.8 is given in order to avoid superfluous repeated methods and contiguous results. We will need the following Lemma:
\begin{lemma}
    Let $S\subset X$, $T\subset Y$. Then at most one of the following two problems has a solution:
    \begin{equation*}
(I) \left. \begin{array}{ll}
         \langle Ax-b,y\rangle\leq 0\;\;\forall y\in T\\
        x\in S \end{array} \right., \;\;\;
(II) \left. \begin{array}{ll}
         A^Ty\in S^*\\
        \langle y,b\rangle <0\\
        y\in T \end{array} \right.
        \end{equation*}
\end{lemma}
\begin{proof}
    Suppose that both $(I)$ and $(II)$ have a solution. Let $\hat{x}\in S$ be a  solution of $(I)$ and $\hat{y}\in T$ be a solution of $(II)$. Then
    \begin{equation*}
        \langle A\hat{x}-b,y\rangle\leq 0\;\;\forall y\in T\Rightarrow\langle A\hat{x},\hat{y}\rangle\leq\langle b,\hat{y}\rangle< 0
    \end{equation*}
    while $\langle A\hat{x},\hat{y}\rangle=\langle\hat{x},A^T\hat{y}\rangle\geq 0$, which is a contradiction.\\
\end{proof}
We can now prove Theorem 3.8:
\begin{proof}\textit{(Theorem 3.8)} Suppose by contradiction that a feasible solution $x\in S$ of $(P)$ such that $Ax=b$ does not exist. By Theorem 3.5 the following problem has a solution:
\begin{equation*}
    (II^*) \left. \begin{array}{ll}
         (-A)^Ty\in S^*\\
        \langle y,-b\rangle <0\\
        y\in Y \end{array} \right.
\end{equation*}
The set $\hat{\mathcal{S}}(D)$ is non-empty, therefore the feasible set $Ax-b\in T^*$, $x\in S$ is non-empty. Due to the relation $T^*\subset($int$T)^*$, $Ax-b\in($int$T)^*$, $x\in S$ has a solution. Hence there exists $x\in S$ such that $\langle (-A)x-(-b),y\rangle\leq 0$ $\forall y\in$ int$T$. By Lemma 3.10 for the sets $S,$ int$T$ the following does not have a solution:
\begin{equation*}
(II) \left. \begin{array}{ll}
         (-A)^Ty\in S^*\\
        \langle y,-b\rangle <0\\
        y\in \mbox{int}T \end{array} \right.
\end{equation*}
Therefore problem $(III)$ has a solution, where problem $(III)$ is defined by:
\begin{equation*}
(III) \left. \begin{array}{ll}
         (-A)^Ty\in S^*\\
        \langle y,-b\rangle <0\\
        y\in Y\smallsetminus\mbox{int}T \end{array} \right.
\end{equation*}
Note that if $\hat{y}\in Y\smallsetminus$int$T$ is a solution of $(III)$ then $\lambda\hat{y}$ is also a solution of the problem $\forall\lambda\in (0,\infty)$ due to \textbf{(F4)}.\par
Now suppose that for every solution $\hat{y}$ of $(III)$ there exists a feasible solution $y^*\in T$ such that $\hat{y}+y^*\in$ int$T$ and $\langle y^*,b\rangle>-\infty$. Then $\hat{y}+y^*$ is also a feasible solution. Indeed, we have
\begin{equation*}
    \langle-A^Ty^*+c,x\rangle\geq 0 \mbox{ and}\;\langle-A^T\hat{y},x\rangle\geq0\;\;\forall x\in S\Rightarrow \langle -A^T(y^*+\hat{y})+c,x\rangle\geq 0\;\;\forall x\in S
\end{equation*}
For solutions $\{n\hat{y}\in Y\smallsetminus$int$T:n\in\mathbb{N}\}$, where $\hat{y}$ is a solution of $(III)$, we obtain $v(D)\rightarrow\infty$ as $n\rightarrow\infty$, which is a contradiction.\par
Therefore there exists a solution $\hat{y}\in Y\smallsetminus$int$T$ of $(III)$ such that $\forall y\in\tilde{\mathcal{S}}(D)$: $\hat{y}+y\in Y\smallsetminus$int$T$. Let $y^*\in\hat{\mathcal{S}}(D)\subset\tilde{\mathcal{S}}(D)$. Then $-A^Ty^*\in S^*$, hence $\langle -A^Tny^*+c,x\rangle\geq 0$ $\forall x\in S$ and $\forall n\in\mathbb{N}$. Alike the above we obtain that $ny^*\in\tilde{\mathcal{S}}(D)$, $\forall y^*\in\hat{\mathcal{S}}(D)$ and $\forall n\in\mathbb{N}$.\par
Since $Y\smallsetminus$int$T$ is a cone and $\hat{y}+ny^*\in Y\smallsetminus$int$T$ $\forall y^*\in\hat{\mathcal{S}}(D)$, $\forall n\in\mathbb{N}$ we obtain that $\hat{y}/n+y^*\in Y\smallsetminus$int$T$ $\forall y^*\in\hat{\mathcal{S}}(D)$, $\forall n\in\mathbb{N}$. As $n\rightarrow\infty$ we get $y^*\in Y\smallsetminus$int$T$ $\forall y^*\in\hat{\mathcal{S}}(D)$, which is a contradiction.
\end{proof}
Again, the proof of Theorem 3.9 is similar to the above, and the dual formulation of Lemma 3.10 is used (the reader may complement the unnecessary details). Theorem 2.2 is obtained as a corollary of the two results:\\
\begin{corollary}\textit{(Theorem 2.2)} If $\hat{\mathcal{S}}(P)\neq\O$, $\hat{\mathcal{S}}(D)\neq\O$, $\tilde{\mathcal{S}}(P)\smallsetminus\hat{\mathcal{S}}(P)\neq\O$, $\tilde{\mathcal{S}}(D)\smallsetminus\hat{\mathcal{S}}(D)\neq\O$ and the optimal values $v(P),v(D)$ are finite, then there exist feasible solutions $\hat{x}\in S$, $\hat{y}\in T$ such that $A\hat{x}=b$, $A^T\hat{y}=c$ respectively and therefore $v(P)=v(D)$ and $\hat{x}$, $\hat{y}$ are optimal.\\ \\
\end{corollary}

\section{Linear programming in Complex space}

The application of the main results in  the trivial case of linear programming in real space is obvious and equivalent results are already known; it suffices to look at any textbook on linear programming and duality theory.\par We will particularly work with conic linear problems that are formulated in complex space and discuss some cases where the main results can be put in an application. Firstly, let the following primal-dual pair of problems:
\begin{center}
\begin{equation*}{(P)\;\;\;\;\;\;\;\;}
\begin{array}{ll}
    min\;\; Re(c,z)\\
    s.t.\;\; Az-b\in T^*\\ \;\;\;\;\;\;\;\;z\in S
    \end{array}
\end{equation*}
\end{center}
\begin{center}
\begin{equation*}{(D)\;\;\;\;\;\;\;\;}
\begin{array}{ll}
    max\;\; Re(w,b)\\
    s.t.\;\; -A^*w+c\in S^*\\ \;\;\;\;\;\;\;\;w\in T
    \end{array}
\end{equation*}
\end{center}

Here the objective function is bilinear, symmetrical and continuous, and is defined by $Re(z,c)=Re(z\cdot c)=Re(z^*c)=Re(\displaystyle\sum_{i=1}^{m}\bar{z_i}c_i)$. The sets $S,T$ are defined by
\begin{equation*}
    S:=\{z\in\mathbb{C}^{m}:|argz|\leqq\alpha\},\;\;\; T:=\{w\in\mathbb{C}^{n}:|argw|\leqq\beta\}
\end{equation*}
where $\alpha$, $\beta$ are real m and n-vectors in $(0,\frac{\pi}{2})e$ respectively, where $e$ are the m and n-vectors with all coordinates equal to $1$. Note that the sets $S,T$ are closed convex cones and that their positive duals are defined by
\begin{equation*}
    S^*=\{z\in\mathbb{C}^{m}:|argz|\leqq\ \frac{\pi}{2}e-\alpha\},\;\;\; T^*=\{w\in\mathbb{C}^{n}:|argw|\leqq \frac{\pi}{2}e-\beta\}
\end{equation*}
Therefore the dual problems can also be writen as
\begin{center}
\begin{equation*}{(P)\;\;\;\;\;\;\;\;}
\begin{array}{ll}
    min\;\; Re(z,c)\\
    s.t.\;\; \lvert argz\rvert\leqq\alpha\\ \lvert arg(Az-b)\rvert\leqq\frac{\pi}{2}e-\beta
    \end{array}
    \end{equation*}
\end{center}
\begin{center}
\begin{equation*}{(D)\;\;\;\;\;\;\;\;}
\begin{array}{ll}
    max\;\; Re(w,b)\\
    s.t.\;\; \lvert argw\rvert\leqq\beta\\ \lvert arg(-A^*w+c)\rvert\leqq\frac{\pi}{2}e-\alpha
    \end{array}
    \end{equation*}
\end{center}
In addition, the linear operator $A$ is a $n\times m$ complex matrix and $A^*$ is its conjugate transpose.\\ \par
Strong duality for the above pair of problems was first proved by Levinson \cite{10}.:
\begin{theorem}
 If $(P),(D)$ have feasible solutions, then they have optimal solutions $\hat{z}\in S$, $\hat{w}\in T$ respectively and $v(P)=v(D)$.
\end{theorem}
 \par 

Conditions \textbf{(F1), (F2)} hold by definition, while the case of assumption \textbf{(F3)} is slightly more complicated and, in agreement to the previous discussion in Sections 1 and 2, is the one that essentially needs to hold in order for the main theorems to be applied. Let us see the following example:\\ \par
Let $\alpha\in(0,\frac{\pi}{2})e$, $\beta\in(0,\frac{\pi}{2})e$, $S=\{z\in\mathbb{C}^{m}:0\leqq argz\leqq\alpha\},\; T=\{w\in\mathbb{C}^{n}:0\geqq argw\geqq-\beta\}$ and all elements $a_{ij}$ of the matrix $A$ satisfy:
\begin{equation*}
    -\frac{\pi}{2}\leq arg(a_{ij})\leq0
\end{equation*}
Then, the objective function $Re(\cdot,\cdot)$ is lower level-bounded over the sets $C_A-b$, $D_A-c$ and as a result \textbf{(F3)} is true (\cite{14} Theorem 1.9) and Theorems $2.1$, $2.2$ can be applied. In particular, we can easily identify the geometric properties of the optimal solutions and get the longed for characterizations analyzed in the previous section:\\ \par
By Theorem $2.1$ if $0<argz<\alpha$ and $0>argw>-\beta$ (for a least one coordinate) then there exist solutions $z^*\in S$, $w^*\in T$ such that $Az=b$ and $A^*w=c$. \par
Now assume that the linear problems $\{Az=b$, $0\leqq argz\leqq\alpha\}$ and $\{A^*w=c$, $0\geqq argw\geqq-\beta\}$ have no solutions. Then, again by Theorem 2.1 we instantly obtain that the optimal solutions $\hat{z},\hat{w}$ belong to the boundary of their corresponfing cones, that is $arg\hat{z}\equiv\alpha$ and $arg\hat{w}\equiv -\beta$ (when the optimal solutions are not real numbers). Therefore, since we already know that the optimal solutions exist when both problems are feasible (Theorem 4.1) and that the two linear systems have no solutions in $S,T$, significant information regarding the geometrical location of these optimal solutions is immediately acquired; when they are not real numbers, they lie on the boundary of the argument cones in their corresponding finite-dimensional complex planes. Note that since \textbf{(F1)-(F3)} hold for cones that are not necessarily closed, this example also holds for cones $S$ and $T$ in which their elements might have coordinates whose argument strictly belongs to $(0,\alpha_j)$ (to $(-\beta_i,0)$ respectively).\par
The latter analysis for this case of complex linear problems however does not work if our goal is to apply the second main result. Obviously, this is due to the fact that all elements of $A$ have a positive real part. One should therefore work in a different setup. For instance, one could work with matrix $A$ such that $|arg(a_{ij})|\leqq\frac{\pi}{2}+min\{\alpha,\beta\}_{k^*}$, $\forall i,j$, with $b>0$ and $c\in S^*$ (why?).\\

\par
The geometric characterization of the optimal solutions of dual problems $(P),(D)$ obtained can have further appliances, such as in game theory. For example, if one works with complex matrix games, as they are defined and analyzed in \cite{4}, Theorem 2.1 implies whether or not the players $I$ and $II$ choose optimal strategies that lie on the boundary of their respective compact strategy sets. Of course, as seen from the proof of the minimax theorem in \cite{4}, the closed cones $S,T$ have a stricter geometrical structure, as they are defined by
\begin{equation*}
    S=\{z\in\mathbb{C}^m:|argz|\leqq\alpha,\;\displaystyle\sum_{i=1}^{m}Im(z_i)=0\},\;\;T=\{w\in\mathbb{C}^n:|argw|\leqq\beta,\;\displaystyle\sum_{j=1}^{n}Im(w_j)=0\}
\end{equation*} \par
Therefore the main results can readily be applied in the case of finite-dimensional linear programs under the three original assumptions, where duality theory already exists and combutational methods for the solution of the complex linear systems $Az=b$, $A^*w=c$ are known. These results however may be applied in the case of infinite linear programming problems in complex space as well, where the above linear systems turn out to be harder to solve.\par For instance, consider the space of all square-summable complex sequences $\ell ^2(\mathbb{C})$ with objective function $Re(z,w)=Re(\displaystyle\sum_{i=1}^{\infty}z_i\bar{w_i})$. Let $b,c\in\ell^2(\mathbb{C})$, $A\in\mathcal{M}(\ell^2(\mathbb{C}))$ and the convex cones:
\begin{equation*}
S:=\{z\in\ell^2(\mathbb{C}):|argz|\leqq\alpha\},\;T:=\{w\in\ell^2(\mathbb{C}):|argw|\leqq\beta\}    
\end{equation*}
where $\alpha\in(0,\frac{\pi}{w})e$, $\beta\in(0,\frac{\pi}{2})e$. We assume that the linear operator $A$ and its adjoint $A^*$ are well defined, that is:
\begin{equation*}
\displaystyle\sum_{i=1}^{\infty}\displaystyle\sum_{j=1}^{\infty}{z_i}\bar{a_{ij}}\bar{w_j}=\displaystyle\sum_{i=1}^{\infty}\displaystyle\sum_{j=1}^{\infty}\bar{z_i}a_{ij}w_j<\infty\;\; \forall z\in S,\;\; \forall w\in T  
\end{equation*}
Conditions \textbf{(F1),\;(F2)} and \textbf{(F4)} hold be definition, while the case of \textbf{(F3)} is again more sensitive. However, when working with specific subsets of the feasible cones and certain cases of the linear operators $A,A^*$, such as in the previous example, the third assumption holds by the known theory of minimum attainment (see \cite{9}). Therefore Theorem 2.1 can be applied. Note that solutions of the problems $\{Az=b$, $z\in S\}$ and $\{A^*w=c$, $w\in T\}$ are not so easily found by similar combutational methods as in the finite case (e.g. with Gauss elimination). Hence the results obtained in Section 3 are instantly more significant and useful for the infinite case.\par
Since some regular and strong duality theorems for conic linear problems in infinite-dimensional Hilbert spaces have already been obtained (see for instance \cite{1, 2}), the existence of optimal solutions can result in valuable mathematical details regarding both the geometric properties of such solutions, that is whether or not they lie on the boundary of the cones $S,T$, as well as the existence and form of solutions of the systems $\{Az=b$, $z\in S\}$ and $\{A^*w=c$, $w\in T\}$.\\ \\

\section{Continuous linear programming}
Another interesting application takes place in continuous linear programming. Here problems of the subsequent forms are discussed:

\begin{center}
\begin{equation*}{(P^*)\;\;\;\;\;\;\;\;}
\begin{array}{ll}
    min\;\; \mathop{\mathlarger{\int_0^T}} x(t)c(t)dt \\
    s.t.\;\;x(t)B(t)\geqq b(t)+ \mathop{\mathlarger{\int_t^T}}x(s)K(s,t)ds\;\;\;\;\;\;f.a.e.\;\;t\in[0,T]\\
    x(t)\geqq 0\;\; f.a.e.\;\;t\in[0,T]\\
    \\
    \end{array}
    \end{equation*}
\end{center}
\begin{center}
\begin{equation*}{(D^*)\;\;\;\;\;\;\;\;}
\begin{array}{ll}
    max\;\; \mathop{\mathlarger{\int_0^T}} y(t)b(t)dt \\
    s.t.\;\;B(t)y(t)\leqq c(t)+ \mathop{\mathlarger{\int_0^t}}K(s,t)y(s)ds\;\;\;\;\;\;f.a.e.\;\;t\in[0,T]\\
    y(t)\geqq 0\;\; f.a.e.\;\;t\in[0,T]\\
    \\
    \end{array}
    \end{equation*}
\end{center}
where $c(t),b(t)$ are bounded and Lebesgue measurable m and n-vectors (functions) respectively, $B(t)$ is an $m\times n$ bounded and Lebesgue measurable matrix, $K(s,t)$ is an $m\times n$ bounded and Lebesgue measurable matrix which is equal to 0 for $s>t$, the functions $x(t),y(t)$ are bounded and Lebesgue measurable, $T$ is finite and $f.a.e.$ $t\in[0,T]$ stands for \textit{``for almost every"} $t\in[0,T]$, that is the inequalities that apart the feasible regions hold for every $t\in[0,T]\smallsetminus U$, where $U\subset[0,T]$ has a Lebesgue measure of zero.\par
Problems of the above form have been discussed in \cite{6, 9} and strong duality theorems have been acquired under specific hypotheses. The following algebraic assumptions are always made:\\
\begin{equation}
\{z:B(t)z\leqq0,z\geqq0\}=\{0\}\;\;\forall t\in[0,T]
\end{equation}
\begin{equation}
    B(t)\geqq0,\;K(s,t)\geqq0,\;c(t)\geqq0\;\;\forall s,t\in[0,T]
\end{equation}
while the continuity of these functions (almost everywhere in $[0,T]$) is also required for no duality gap between $(P^*)$ and $(D^*)$. In particular, the following strong duality theorem is known:

\begin{theorem}[\cite{6} Theorem 3.4, \cite{9} Theorem 3] If (5.1) and (5.2) hold and if $B(t),\;b(t),\;c(t),\;K(s,t),\;$ are continuous at almost all $t$ in $[0,T]$ and almost all $s,\;t$ in $[0,T]\times[0,T]$ respectively, then $v(P^*)=v(D^*)$.
\end{theorem}
\par
 Here we work with the conves cones:
\begin{equation*}
    S:=(L_{\infty}^{+}[0,T])^m,\;\;\;T:=(L_{\infty}^{+}[0,T])^n
\end{equation*}
where $L_{\infty}^{+}[0,T]$ is the space of all almost everywhere positively valued on $[0,T]$ bounded Lebesgue measurable functions. Strong duality results regarding continuous linear problems of similar form with these exact feasible convex cones are given in \cite{15}. \par The linear operators $A:X\rightarrow Y$, $A^T:Y\rightarrow X$ are defined by:
\begin{equation*}
    (Ax)(t):= x(t)B(t)-\mathop{\mathlarger{\int_t^T}}x(s)K(s,t)ds,\;\;t\in[0,T]
\end{equation*}
\begin{equation*}
    (A^Ty)(t):=B(t)y(t)-\mathop{\mathlarger{\int_0^t}}K(s,t)y(s)ds,\;\;t\in[0,T]
\end{equation*}
We can see that these are well defined under our initial hypothesis. Also for the bilinear and symmetric objective function $\langle z(t),w(t)\rangle = \mathop{\mathlarger{\int_0^T}}z(t)w(t)dt$ the following is true: $\langle (Az)(t),w(t)\rangle=\langle z(t),(A^Tw)(t)\rangle$. Indeed, it suffices to show that:\\
\begin{equation*}
\mathop{\mathlarger{\int_0^T}}\left(\mathop{\mathlarger{\int_t^T}}z(s)K(s,t)ds\right)w(t)dt=\mathop{\mathlarger{\int_0^T}}z(t)\left(\mathop{\mathlarger{\int_0^t}}K(s,t)w(s)ds\right)dt
\end{equation*}
This is true by the Fubini Theorem (see \cite{6}, Proposition 1.3).\par
Now \textbf{(F1),(F2),(F4)} hold by definition for almost every $t\in[0,T]$ (it is easy to notice that the \textit{``for almost every"} does not affect the gist of the main results - the reader may fill in the appropriate details). Condition \textbf{(F3)} is the one that discommodes us once again. In contrast to the previous example, it is not wise to work with $L_{\infty}^{+}[0,T]$ (or a subspace $L_{k}^{+}[0,T]$ for some $k\in\mathbb{N}$) in order to prove level-boundness and/or apply Theorems of the attainment of minimum, as it was done previously. This is mainly because the set $\{f(t)\in (L_{\infty}^{+}[0,T])^m: |f(t)|\leqq N, t\in[0,T]\}$ for some $N\in\mathbb{R}$ is not bounded. Therefore we would either have to work in specific subspaces of $L_{\infty}^{+}[0,T]$, that is with specific spaces - types - of bounded Lebesgue measurable functions, or with suitable matrices $K(s,t)$, $B(t)$ and vectors $b(t)$, $c(t)$, as it was effectively done in the previous section.\par
For the main purposes of this section we shall not deal with the analysis of cases under which this condition is met; for the result that follows we will assume that \textbf{(F3)} is true.\\ \par The next Theorem can be characterized as partially complementary to Theorem 5.1:\\

\begin{theorem}
Let $\hat{x}(t)$ be a feasible solution of $(P^*)$ and $\hat{y}(t)$ be a feasible solution of $(D^*)$ such that $\hat{x}(t)>0$ and $\hat{y}(t)>0$ f.a.e.\;\;$t\in[0,T]$. Assume that the optimal values $v(P^*)$, $v(D^*)$ exist and are bounded, assumption (F3) holds and that either $(i)$ or $(ii)$ of the following is satisfied:\\ \\
(i) $B(t)\leqq0$, $K(s,t)\geqq0$ $\forall s,t\in[0,T]$ and $b\in T^*$\\
(ii) $B(t)\geqq0$, $K(s,t)\leqq0$ $\forall s,t\in[0,T]$ and $-c\in S^*$\\ \\
Then there exist optimal solutions $x^*(t)$ of $(P^*)$ and $y^*(t)$ of $(D^*)$ such that $x^*(t)B(t)= b(t)+ \int_t^T x^*(s)K(s,t)ds$ and $B(t)y^*(t)= c(t)+ \int_0^t K(s,t)y^*(s)ds$,\;\;\;f.a.e. $t\in[0,T]$, and $v(P^*)=v(D^*)$.
\end{theorem}
\begin{proof}
    Assume that the first condition $(i)$ is met (the proof for condition $(ii)$ is completely analogous). Then, since $B(t)\leqq0$ and $K(s,t)\geqq0$ for $s,t\in[0,T]$, the set $\hat{\mathcal{S}}(D^*)=\{y\in$ int$T:-A^Ty+c\in S^*,-A^Ty\in S^*,\langle y,b\rangle>-\infty\}$, where the linear operator $A^T$ is defined as above, is nonempty for almost every $t\in[0,T]$. Therefore, by Theorem 3.8 there exists $x^*(t)\in\mathcal{S}(P^*)$ such that $Ax^*=b$, that is $x^*(t)B(t)= b(t)+ \int_t^T x^*(s)K(s,t)ds$ for almost every $t\in[0,T]$. Now, since $b\in T^*$, Theorem 3.9 may also be applied, hence there also exists $y^*(t)\in\mathcal{S}(D^*)$ such that $A^Ty^*=c$, that is $B(t)y^*(t)= c(t)+ \int_0^t K(s,t)y^*(s)ds$,\;\;\;f.a.e. $t\in[0,T]$. The strong duality relation $v(P^*)=v(D^*)$ follows.
\end{proof}
Note that for the proof the continuity of $B(t),\;b(t),\;c(t),\;K(s,t)$ is not assumed in comparison to Theorem 5.1, although it may be (partially) needed in order for \textbf{(F3)} to hold.

\bibliographystyle{plain}

\end{document}